\documentclass[smallextended,natbib,runningheads]{svjour3}
\journalname{Annals of the Institute of Statistical Mathematics}
\def\makeheadbox{}
\smartqed  % flush right qed marks, e.g. at end of proof
\usepackage{siunitx} % Required for alignment
\usepackage{amsmath,amssymb}
\usepackage{graphicx,tikz,url,array,dsfont}
\usetikzlibrary{arrows,shapes}
\usetikzlibrary{decorations.pathreplacing}
\usetikzlibrary{decorations.pathmorphing}
\tikzset{discont/.style={decoration={zigzag,segment length=12pt, amplitude=4pt},decorate}}
\def\discontarrow(#1)(#2)(#3)(#4);{
  \draw[discont] (#2) -- (#3);
  \draw[->] (#1) -- (#2) (#3) -- (#4);}
\usetikzlibrary{external}
\usepackage{booktabs}
\usepackage{color}

\newcommand{\I}{\mathcal{I}}
\newcommand{\E}{E}
\newcommand{\R}{\mathbb{R}}
\newcommand{\reals}{\mathbb{R}}
\newcommand{\N}{\mathbb{N}}

\renewcommand{\P}{\mathcal{P}}
\newcommand{\subP}{\P_{\psi,\varepsilon}}
\newcommand{\QQ}{\mathcal{Q}}
\newcommand{\Y}{\mathcal{Y}}

\newcommand{\mode}{\Gamma_{\text{mode}}}
\newcommand{\heinrichp}{p_{\text{mix}}}

\newcommand{\argmin}{\mathop{\mathrm{argmin}}}

\newtheorem{myclaim}{Claim}

\begin{document}

\title{On the Indirect Elicitability of the Mode and Modal Interval\thanks{This work was supported by National Science Foundation Grant CCF-1657598.}}

%\titlerunning{The Mode is not Indirectly Elicitable}        % if too long for running head

\author{Krisztina Dearborn \and Rafael Frongillo}


\institute{K. Dearborn \at
              Department of Mathematics, University of Colorado Boulder \\ Campus Box 395, Boulder, CO, USA 80309  \\
              \email{krisztina.dearborn@colorado.edu}                      \and
           R. Frongillo (corresponding author)\at
              Department of Computer Science, University of Colorado Boulder\\ 1111 Engineering Dr, Boulder, CO, USA 80309\\
              \email{raf@colorado.edu}
}

\date{Received: date / Revised: date}

\maketitle

\begin{abstract}
Scoring functions are commonly used to evaluate a point forecast of a particular statistical functional.
  This scoring function should be consistent, meaning the correct value of the functional is the Bayes act, in which case we say the scoring function elicits the functional.
  Recent results show that the mode functional is not elicitable.
  In this work, we ask whether it is at least possible to indirectly elicit the mode, wherein one elicits a low-dimensional functional from which the mode can be computed.
  We show that this cannot be done: neither the mode nor a modal interval are indirectly elicitable with respect to the class of identifiable functionals.
\end{abstract}

\section{Introduction}\label{intro}
To evaluate point forecasts, one commonly uses a scoring function, also called a loss function, which measures the inaccuracy of the forecast relative to an observed outcome.
Loss functions are also used in estimation, forecast ranking and comparison, model selection, and back-testing~\citep{gneiting2007strictly,gneiting}.
In all of these applications,
given a target statistical functional, we desire a consistent loss function, meaning that the correct value of the functional is the Bayes act with respect to the loss.
In this case, we say the loss function \emph{elicits} the functional.

While many common statistics are elicitable, such as the mean, median, and quantiles, it is well-known that the variance is not.
This impossibility follows from an observation of Osband~\citeyearpar{osband} that elicitable functionals have convex level sets, meaning mixtures of distributions with the same functional value must again have the same functional value, an axiom which the variance does not satisfy.
(Indeed, mixtures generally have higher variance.)
Nonetheless, several authors have pointed out that the variance is \emph{indirectly} elicitable: one may elicit the first and second moment of the distribution, and then combine these values with a link function to obtain the variance.
The minimum number of dimensions required in such an indirect elicitation scheme (for the variance, 2) is referred to as the elicitation order, or elicitation complexity, of the functional in question~\citep{lambert08}.
Recently, several important non-elicitable functionals, including many risk measures such as conditional value at risk, have been shown to be indirectly elicitable with low elicitation complexity~\citep{lambert08,frongillo_kash,fissler2016higher}.

Heinrich~\citeyearpar{heinrich} recently showed that another common statistic, the mode functional, is not elicitable, despite the fact that its level sets are convex.
It is therefore natural to ask whether the mode is indirectly elicitable, and if so, determine its elicitation complexity.
Our main result is that the mode has infinite elicitation complexity with respect to identifiable functionals, a relatively weak restriction (see Definitions~\ref{def:iden} and~\ref{def:elic-complex} and the discussion following).
% Following \citet{frongillo_kash}, we restrict to functionals which are identifiable, a notion due to Osband~\citeyearpar{osband} which is akin to Z-estimation, and we show that the mode is not indirectly elicitable with respect to this class.
% The restriction to identifiability is relatively weak, as discussed following Definition~\ref{def:iden-complex}.
Interestingly, our results also extend to modal intervals, which are elicitable, as we discuss in Section~\ref{sec:modal-interval}.

Our results show that it is impossible to develop a consistent loss function for evaluating point forecasts of the mode, even indirectly.
Moreover, they cast doubt on the existence of broadly effective empirical risk minimization schemes for estimating the mode or a modal interval.
Our techniques differ from previous work~\citep{frongillo_kash}, and may be applicable to other functionals of interest.
We conclude with open questions, including a discussion of other notions of elicitation complexity and other properties.

\section{Setting}\label{back}
Let $\P$ be a set of probability measures on a common measurable space $(\Y, \mathcal{F})$\@.
For each probability measure $P\in\P$, denote the expectation of a random variable $Y$ with distribution $P$ by $E_P Y$\@.
We will use the term ``property'' to refer to a statistical functional taking values in a report space $\mathcal{R}$, often a subset of $\reals$ or $\reals^k$.

\begin{definition}[Property]
A property is a functional $\Gamma:\P\to \mathcal{R}$ which assigns a report value to each probability measure in $\P$\@.
\end{definition}

For example, considering probability measures on the measurable space $\Y=\R$, with $\mathcal{F}$ being the Borel $\sigma$-algebra on $\Y$, the mean $\Gamma(P)=E_P(Y)$ is a real-valued property\@.
Similarly, another real-valued property is the variance, $\Gamma(P)=E_P(Y-E_P(Y))^2$\@.
Our focus in this paper will be the mode, which will be defined with care at the end of this section\@.

We next formalize our notion of consistency, which ensures that the Bayes act for a loss function coincides with the desired property.

\begin{definition}[Elicits]
  \label{def:elic}
  A loss function $L:\mathcal{R}\times\Y\to \R$ elicits a property $\Gamma:\P\to\mathcal{R}$  if for every $P\in \P$ we have $\displaystyle\{\Gamma(P)\}=\argmin\nolimits_rE_P(L(r,Y))$\@.
 We say $\Gamma$ is elicitable if there exists some loss function that elicits $\Gamma$.
 For all $k\in\N$ the set of elicitable properties $\Gamma:\mathcal{P}\to\R^k$ will be denoted $\mathcal{E}_k(\mathcal{P})$\@.
\end{definition}

\noindent
For example, the mean is elicited by squared loss $L(r, y)=(r-y)^2$\@.

The following concept of identifiability, due to Osband~\citeyearpar{osband}, has played a central role in the theory of property elicitation\@.
The definition states that each level set of the property, that is, the set of distributions sharing a particular value $r$ of the property, can be described by a linear constraint which depends on $r$.
Note that Steinwart et al.~\citeyearpar{steinwart} adopt a weaker notion of identifiability, wherein the condition need only hold for almost every level set, and call the definition below ``strong identifiability''; see Section~\ref{dis}.
\begin{definition}[Identification]
  \label{def:iden}
A property $\Gamma:\P\to\mathcal{R}\subseteq \R^k$ is identifiable if there exists an identification function $V:\mathcal{R}\times \Y\to \R^k$ such that for all $P\in \P$ we have $\Gamma(P)=r$ if and only if $E_P (V(r, Y))={0}$\@.
Let $\mathcal{I}_k(\mathcal{P})$ denote the class of all properties from $\mathcal{P}$ to $\R^k$ which are identifiable\@.
\end{definition}

\noindent
To illustrate, the mean is identified by the function $V(r, y)=y-r$.

Let us return to the notion of elicitation, and consider the variance $\Gamma(P) = E_P(Y- E_P(Y))^2$.
As observed by Osband~\citeyearpar{osband}, for a property to be elicitable it must have convex level sets: the set of distributions having the same property value must be convex\@.
It follows immediately that the variance is not elicitable\@.
As noted in the introduction, however, the variance can be expressed as a function, or link, of elicitable properties, for example the mean and second moment: $\Gamma(P) = E_P(Y^2)- ( E_P(Y))^2$.
This motivates the notion of indirect elicitability, wherein one elicits an intermediate property, and then computes a link function to obtain the original property\@.
When confronted with non-elicitable properties, it is therefore natural to ask the minimal dimension of such an intermediate elicitable property; this is the notion of elicitation complexity~\citep{lambert08,frongillo_kash,fissler2016higher}\@.
As we explain following the definition, we further require that these intermediate properties be identifiable.

\begin{definition}[Identifiable Elicitation Complexity]
 \label{def:elic-complex}
  Let $\I = \bigcup_{k\in\N} \I_k(\P)$ be the class of identifiable properties.
For $k\in \N$, a property $\Gamma:\P\to \mathcal{R}$ is $k$-elicitable with respect to $\I$ if there exists an elicitable property $\hat{\Gamma}\in\mathcal{E}_k(\mathcal{P})\cap\I$ and a function $f:\R^k\to \mathcal{R}$ such that $\Gamma=f\circ \hat{\Gamma}$\@.
The identifiable elicitation complexity of $\Gamma$ is then the minimum of all $k$ such that $\Gamma$ is $k$-elicitable with respect to $\I$\@.
\end{definition}

Without imposing such a restriction on the class of intermediate properties, the definition of elicitation complexity would be trivial, as noted by Frongillo \& Kash~\citeyearpar{frongillo_kash}: all properties of distributions on $\reals$ have complexity 1 by first eliciting the entire distribution via set-theoretic bijections between $\R$ and $\R^\N$ (see also the discussion following Corollary~\ref{cor:modal})\@.
To justify the restriction to identifiability in particular, first note that nearly all natural elicitable properties are identifiable, including expectations, ratios of expectations, quantiles, and expectiles.
Second, the results of Lambert~\citeyearpar{lambert18} and Steinwart et al.~\citeyearpar{steinwart} show that continuous non-locally-constant functionals are elicitable if and only if they are weakly identifiable, meaning identifiability is essentially necessary for continuous non-locally-constant properties $\hat{\Gamma}$ in Definition~\ref{def:elic-complex}.
Third, elicitable properties which are not identifiable are often indirectly elicitable via finite-dimensional identifiable properties, as is the case for all finite elicitable properties (those taking values in a finite set); this observation is particularly relevant as we give infinite lower bounds.

Returning to the example of the variance, we see that while it is not elicitable, its identifiable elicitation complexity is at most~$2$\@.
The variance can be recovered via the function $f(x_1, x_2)=x_2-x_1^2$ composed with the identifiable and elicitable vector-valued property $\hat{\Gamma}(P)=(E_P(Y), E_P(Y^2)) \in \R^2$\@.
In this case the identification function for $\hat{\Gamma}$ is $V(r,y)=(y-r_1, y^2-r_2)$ where $r=(r_1,r_2)$\@.
There is a distinction between a property which is elicitable like the mean, $\Gamma(P)=E_P(Y)$, and a property which is $1$-elicitable like the mean squared, $\Gamma(P)=(E_P(Y))^2$.
While every elicitable real-valued property is trivially $1$-elicitable via the identity function, not every $1$-elicitable property is elicitable\@.
The mean squared fails to be elicitable, but is $1$-elicitable\@.

Finally, we define identification complexity, which trivially lower bounds identifiable elicitation complexity, a fact we use extensively in our results\@.

\begin{definition}[Identification Complexity]
  \label{def:iden-complex}
For $k\in \N$, a property $\Gamma:\P\to\mathcal{R}$ is $k$-identifiable if there exists an identifiable property $\hat{\Gamma}\in\mathcal{I}_k(\mathcal{P})$ and a function $f:\R^k\to \mathcal{R}$ such that $\Gamma=f\circ \hat{\Gamma}$\@.
Furthermore, the identification complexity of $\Gamma$ is the minimum of all $k$ such that $\Gamma$ is $k$-identifiable\@.
\end{definition}

For the remainder of this section, we turn to the mode, which we define as in Gneiting~\citeyearpar{gneiting} and Heinrich~\citeyearpar{heinrich}\@.
Letting $\varepsilon>0$ and $P\in \P$, consider the cumulative distribution function $F$ associated with $P$\@.
A modal interval is any interval of the form $[x-\varepsilon,x+\varepsilon]$ to which $F$ assigns maximal probability\@.
Let $\Gamma_\varepsilon$ denote a midpoint of a modal interval, defined as
\begin{equation}
  \label{eq:modal-midpoint}
  \Gamma_\varepsilon(P) \in \arg \max_x \; \left( F(x+\varepsilon)-\lim_{z\uparrow x-\varepsilon}F(z) \right)~.
\end{equation}

Regardless of whether the modal interval is unique we can use its midpoint to define the mode of the distribution\@.
Suppose there exists a sequence of real numbers $\{\varepsilon_n\}$ where $\varepsilon_n\to 0$ as $n\to\infty$ and a corresponding choice of midpoints of modal intervals $\{\Gamma_{\varepsilon_n}(P)\}$ converging to a real number, $\mode(P)$. Then $\mode(P)$ is the mode of the distribution\@.
This definition is careful not to assume that a probability density exists\@.
In the case where the distribution function $F$ is absolutely continuous and admits a continuous density $p$, then $\mode(P)$ coincides with the global maximum of $p$\@.
When working with a discrete probability distribution, $\mode(P)$ corresponds to the point(s) associated with maximal probability\@.

We will refer to probability measures which have a well-defined and unique mode as \emph{unimodal}\@.
If a probability measure is unimodal and there exists a probability density associated with it, the density does not necessarily have a unique local maximum, a stronger requirement\@.
For example, a Gaussian density is unimodal in both senses of the term,
whereas a mixture of Gaussians with unit variance and strictly distinct weights does not necessarily have a unique local maximum, but does have a well-defined and unique mode and thus is unimodal\@.
(See Section~\ref{dis} for a discussion of the stronger definition\@.)

\section{Impossibility}\label{results}
Heinrich~\citeyearpar{heinrich} demonstrates that the mode is not directly elicitable with respect to several classes of unimodal probability measures.
We proceed by studying the identifiable elicitation complexity of the mode\@.
Our main results Theorems~\ref{thm:bounded-density} and~\ref{thm:gaussians} both show that the identifiable elicitation complexity of the mode is infinite with respect to two classes of probability measures\@.
These results imply that, when restricting to identifiable intermediate properties, the mode is not even indirectly elicitable.
% Corollary~\ref{cor:modal} in turn shows that for any $\varepsilon>0$, the modal midpoint $\Gamma_\varepsilon$ also has infinite identifiable elicitation complexity\@.

To begin, let $\P$ denote the class of unimodal probability measures defined on the real line which admit a smooth and bounded density\@.
Below we will define a class $\subP$ of probability measures within $\P$ consisting of (finite) mixtures of normalized bump functions which will be the class of probability measures employed in Lemma~\ref{lem:not-iden}, Theorem~\ref{thm:bounded-density}, and Corollary~\ref{cor:modal}\@.
We will denote by $\QQ \subset \P$ the class of probability measures which can be expressed as a (finite) mixture of Gaussians, the focus of Theorem~\ref{thm:gaussians}\@.
Since each $P\in\P$ admits a unique probability density $p$, we will identify the probability measure $P$ with its density $p$, and use the two interchangeably\@.
Hence, when we choose an element $p\in \P$, we mean the probability density $p$ associated with a probability measure $P\in\P$\@.
Finally, $\mode(P)$ and $\mode(p)$ both denote the mode of the distribution $P$ as defined in Section~\ref{back} which corresponds to the global maximum of $p$\@.

We define the bump function centered at $0$ of width $2\varepsilon>0$ as follows,
\begin{equation}
  \label{eq:bump}
  \psi_{0,\varepsilon}(x)= \begin{cases}
   \tfrac 1 {c_\varepsilon} \exp\left(-\frac{1}{\varepsilon^2-x^2}\right) & |x| < \varepsilon\\
    0 & |x| \geq \varepsilon~,
  \end{cases}
\end{equation}
where $c_\varepsilon=\int_{-\varepsilon}^\varepsilon \exp(-1/(\varepsilon^2-x^2))~ dx$\@.
We then define the bump centered at $x_0$ to be the function $\psi_{x_0,\varepsilon}(x) = \psi_{0,\varepsilon}(x-x_0)$.
Note that $\psi_{x_0,\varepsilon} \in \P$ and $\mode(\psi_{x_0,\varepsilon}) = x_0$.
Let $\subP$ denote the class of distributions in $\P$ which are finite mixtures of bump functions in the set $\{\psi_{4t\varepsilon,\varepsilon} : t\in\mathbb{N}\}$, i.e., of width $2\varepsilon$ centered at $\{0,4\varepsilon,\dots, 4(t-1)\varepsilon,4t\varepsilon,\dots\}$\@.

To build intuition, let us first see why the mode itself is not identifiable.
In fact, we will establish the stronger statement that the mode is not identifiable with respect to $\subP \subset \P$.
(See also \cite[Lemma 2.4]{fissler2017order}.)

\begin{lemma}\label{lem:not-iden}
  The mode, $\mode:\P\to \mathcal{R}$, is not identifiable with respect to $\P$, the class of unimodal probability measures defined on the real line which admit a smooth and bounded density\@.
\end{lemma}
\begin{proof}
  For a contradiction, suppose there exists $V:\R\times\reals\to\reals$ such that $\mode$ is identified by $V$.
  For $h = 2/3$ define the density $p_h = h \psi_{0,1} + (1-h)\psi_{4,1}$ in $\P_{\psi,1}$.
  Clearly, $\mode(p_h) = \mode(\psi_{0,1}) = 0$, and since $V$ identifies $\mode$, we thus have $\E_{p_{h}}V(0,Y) = 0$ and $\E_{\psi_{0,1}} V(0,Y) = 0$.
  Combining,
  \[ 0 = \E_{p_h}V(0,Y) = h\E_{\psi_{0,1}} V(0,Y) + (1-h)\E_{\psi_{4,1}} V(0,Y) = (1-h)\E_{\psi_{4,1}} V(0,Y)~,\]
  from which we conclude $\E_{\psi_{4,1}} V(0,Y) = 0$ and thus $\mode(\psi_{4,1}) = 0$, a contradiction.
\end{proof}

We now see that the mode is not identifiable, but it remains to understand its identifiable elicitation complexity\@.
Theorem~\ref{thm:bounded-density} generalizes the argument of Lemma~\ref{lem:not-iden}, showing that the mode is not \emph{indirectly} identifiable with respect to $\P$, the class of unimodal probability measures defined on the real line which admit a smooth and bounded density.
In other words, for this class $\P$, there is no way to express the mode as a function of a finite-dimensional identifiable property.
We conclude that the identifiable elicitation complexity of the mode is infinite with respect to $\P$\@.

\begin{theorem} \label{thm:bounded-density} 
The mode, $\mode$, has infinite identifiable elicitation complexity with respect to $\P$, the class of unimodal probability measures defined on the real line which admit a smooth and bounded density\@.
\end{theorem}

We briefly outline the proof of Theorem~\ref{thm:bounded-density}; the full proof appears in Appendix~\ref{app:proofs}.
Let $V$ be an identification function, which identifies some intermediate property $\hat\Gamma:\P\to\reals^k$ for some finite dimension $k$.
Taking $t>k$, we construct a probability density $p\in \subP$ with bump heights specified by a vector $h\in\reals^{t+1}_+$, chosen so that the gap between any two bump heights is smaller than the minimum height.
We observe that the expected value of $V$ is linear in the bump heights, and moreover this linear transformation is rank deficient, giving us a nontrivial vector $h'\in\reals^{t+1}$ in its kernel.
By our initial choice of $h$, for any such $h'$ we can find a suitable choice of coefficient $\alpha \in \reals$ so that $h + \alpha h' \in \reals^{t+1}_+$ while changing the mode.
After normalization, this gives us a valid density $p'\in\subP$ yielding zero expectation of $V$, and thus residing in the same level set of $\hat\Gamma$ as $p$, yet with a different mode.
This contradicts the existence of a function $f$ satisfying $\mode=f\circ \hat{\Gamma}$, as $f$ would need to map the same $\hat\Gamma$ value to two different $\mode$ values\@.
Figure~\ref{fig2} illustrates this construction, showing the density $p$ along with a hypothetical choice of $p'$\@.

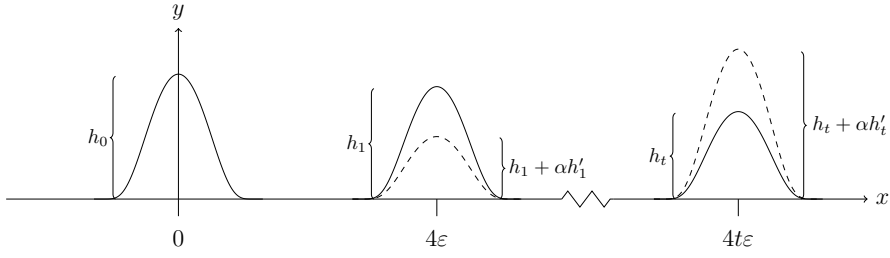
\begin{figure}
  \centering
  \resizebox {\textwidth} {!} {
    \begin{tikzpicture}
      \draw[->] (0,0) -- (0,1) node [above]
      {$y$};
      \begin{scope}[smooth,draw=black,y=0.3989422804cm]
        \draw[black] plot[id=f7,domain=-0.5:0.5,samples=100]
        function {100*exp(-1/(0.25-x*x))};
        \draw [decorate,decoration={brace},xshift=-1pt,yshift=0pt]
        (-0.33,0) -- (-0.33,1.8) node [black,midway,xshift=-0.3cm] 
        {\footnotesize $h_0$};
        
        \draw[black] plot[id=f7,domain=1:2,samples=100]
        function {90*exp(-1/(0.25-(x-1.5)*(x-1.5)))};
        \draw [decorate,decoration={brace},xshift=-1pt,yshift=0pt]
        (1.17,0) -- (1.17,1.62) node [black,midway,xshift=-0.3cm] 
        {\footnotesize $h_1$};
        \draw[black,dashed] plot[id=f7,domain=1:2,samples=100]
        function {50*exp(-1/(0.25-(x-1.5)*(x-1.5)))};
        \draw [decorate,decoration={brace,mirror},xshift=-1pt,yshift=0pt]
        (1.9,0) -- (1.9,0.9) node [black,midway, xshift=0.85cm]
        {\footnotesize $h_1+\alpha h_1'$};
        
        \draw[black] plot[id=f7,domain=2.75:3.75,samples=100]
        function {70*exp(-1/(0.25-(x-3.25)*(x-3.25)))};
        \draw [decorate,decoration={brace},xshift=-1pt,yshift=0pt]
        (2.92,0) -- (2.92,1.26) node [black,midway,xshift=-0.3cm] 
        {\footnotesize $h_t$};
        \draw[black,dashed] plot[id=f7,domain=2.75:3.75,samples=100]
        function {120*exp(-1/(0.25-(x-3.25)*(x-3.25)))};
        \draw [decorate,decoration={brace,mirror},xshift=-1pt,yshift=0pt]
        (3.65,0) -- (3.65,2.16) node [black,midway,xshift=0.85cm]
        {\footnotesize $h_t+\alpha h_t'$};
      \end{scope}
      \draw[-] (-1,0) -- (0,0) node [right] {};
      \discontarrow(0,0)(2.225,0)(2.525,0)(4,0);
      \draw[.](4,0) node [right]{$x$};
      \foreach \pos/\label in {0/$0$,
        1.5/$4\varepsilon$,3.25/$4t\varepsilon$}
      \draw (\pos,0) -- (\pos,-0.1) (\pos cm,-1.5ex) node
      [anchor=base,fill=white,inner sep=1pt]  {\label};
    \end{tikzpicture}
  }
  \caption{The initial density $p$ of Theorem~\ref{thm:bounded-density} depicted with a solid line, and alternate density $p'$ (before normalization) with a dashed.  Here $t > k$, where $k$ is the dimension of the intermediate property $\hat\Gamma$.}
  \label{fig2}
\end{figure}

The impossibility result of Theorem~\ref{thm:bounded-density} is strengthened in Theorem~\ref{thm:gaussians}, which shows that the mode has infinite identifiable elicitation complexity even after restricting to the family $\QQ$ of probability measures in $\P$ which can be expressed as a mixture of Gaussians\@.
While the general outline of the proof is similar, the bump functions used in Theorem~\ref{thm:bounded-density} were supported on disjoint intervals, which is clearly not true of Gaussians\@.
In particular, changing the heights of distant Gaussians will now alter the mode\@.

\begin{theorem} \label{thm:gaussians}
  The mode, $\mode:\QQ\to \mathcal{R}$, has infinite identifiable elicitation complexity with respect to $\QQ$, the class of probability measures in $\P$ which can be expressed as a mixture of Gaussians\@.
\end{theorem}

See Appendix~\ref{app:proofs} for the proof, which shows that the statement holds even when the mixture is over Gaussians with the same variance.
As in Theorem~\ref{thm:bounded-density}, we assume $\mode = f \circ \hat\Gamma$ for some finite-dimensional identifiable $\hat\Gamma$, construct an initial density $q\in\QQ$, and show that there must exist another $q'\in\QQ$ in the same level set as $q$ but with a different mode (see Figure~\ref{fig3}).
While the broad outline remains the same, several technical challenges arise from the overlapping supports of Gaussians.
To address these issues, we bound the potential contribution of one Gaussian to the density value at another to show that the mode changes from $q$ to $q'$, and use these bounds again to set the height vector $h$ so that a coefficient $\alpha$ still exists for all possible vectors $h'$.

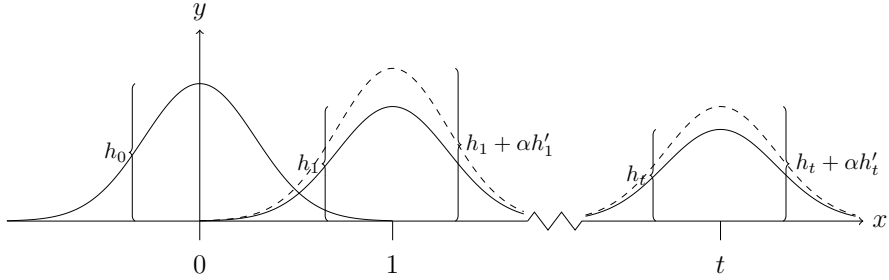
\begin{figure}
  \centering
  \resizebox {\textwidth} {!} {
    \begin{tikzpicture}
      \def\myprecision{6}
      \draw[->] (0,0) -- (0,1) node [above]
      {$y$};
      \begin{scope}[smooth,draw=black,y=0.3989422804cm]
        \draw[black] plot[id=f7,domain=-1:1,samples=100]
        function {1.8*exp(-\myprecision*x*x)};
        \draw [decorate,decoration={brace},xshift=-1pt,yshift=0pt]
        (-0.3,0) -- (-0.3,1.8) node [black,midway,xshift=-0.3cm] 
        {\footnotesize $h_0$};
        
        \draw[black] plot[id=f7,domain=0:1.68,samples=100]
        function {1.5*exp(-\myprecision*(x-1)*(x-1))};
        \draw [decorate,decoration={brace},xshift=-1pt,yshift=0pt]
        (0.7,0) -- (0.7,1.5) node [black,midway,xshift=-0.3cm] 
        {\footnotesize $h_1$};
        \draw[black,dashed] plot[id=f7,domain=0:1.68,samples=100]
        function {2*exp(-\myprecision*(x-1)*(x-1))};
        \draw [decorate,decoration={brace,mirror},xshift=-1pt,yshift=0pt]
        (1.36,0) -- (1.36,2) node [black,midway, xshift=0.85cm]
        {\footnotesize $h_1+\alpha h_1'$};
        
        \draw[black] plot[id=f7,domain=2:3.4,samples=100]
        function {1.2*exp(-\myprecision*(x-2.7)*(x-2.7))};
        \draw [decorate,decoration={brace},xshift=-1pt,yshift=0pt]
        (2.4,0) -- (2.4,1.2) node [black,midway,xshift=-0.3cm] 
        {\footnotesize $h_t$};
        \draw[black,dashed] plot[id=f7,domain=2:3.4,samples=100]
        function {1.5*exp(-\myprecision*(x-2.7)*(x-2.7))};
        \draw [decorate,decoration={brace,mirror},xshift=-1pt,yshift=0pt]
        (3.06,0) -- (3.06,1.5) node [black,midway,xshift=0.85cm]
        {\footnotesize $h_t+\alpha h_t'$};
      \end{scope}
      \draw[-] (-1,0) -- (0,0) node [right] {};
      \discontarrow(0,0)(1.7,0)(2,0)(3.44,0);
      \draw[.](3.44,0) node [right]{$x$};
      \foreach \pos/\label in {0/$0$,
        1/$1$,2.7/$t$}
      \draw (\pos,0) -- (\pos,-0.1) (\pos cm,-1.5ex) node
      [anchor=base,fill=white,inner sep=1pt]  {\label};
    \end{tikzpicture}
  }
  \caption{ 
    The initial density $q$ in Theorem~\ref{thm:gaussians} depicted as a mixture of Gaussians with solid lines, and the alternate density $q'$ with dashed, before normalization.}\label{fig3}
\end{figure}

\section{Implications for the Modal Interval}\label{sec:modal-interval}

While the mode is not elicitable, it is well-known that the midpoint of the modal interval $\Gamma_\varepsilon$ defined in eq.~\eqref{eq:modal-midpoint}, which we will refer to as the \emph{modal midpoint}, is elicitable, via the simple loss function $L_\varepsilon(r,y)=\mathds{1}\{|r-y|>\varepsilon\}$.
(Note that as we restrict to single-valued functionals in this paper, in the technical results that follow, we will only consider distributions with a unique solution to eq.~\eqref{eq:modal-midpoint}.)
Recalling that the mode is the limit of the modal midpoint $\Gamma_\varepsilon$ as the radius $\varepsilon$ approaches $0$, it is often suggested to estimate the mode by $\Gamma_\varepsilon$ for a sufficiently small $\varepsilon$.
Heinrich~\citeyearpar{heinrich} argues that this practice is ill-advised, given the non-elicitability of the mode, and further demonstrates this argument empirically.
For a particular Gaussian mixture with density $\heinrichp$, with two local maxima $m_0$ and $m_1$, $\heinrichp(m_0) < \heinrichp(m_1)$, Heinrich shows that given a fixed number of samples, even small values of $\varepsilon$ result in a modal midpoint $\hat x_\varepsilon$ which is more often closer to $m_0$ than the mode $m_1$.
More precisely, for $\varepsilon \in \{0.5, 0.25, 0.1, 0.05, 0.025, 0.001\}$, out of 1000 trials each, Heinrich finds that $|\hat x_\varepsilon - m_1| < |\hat x_\varepsilon - m_0|$ in no more than 438 trials.
Moreover, this success rate drops when $\varepsilon < 0.1$.

\begin{figure}
  \centerline{\includegraphics[width=1.2\textwidth,clip]{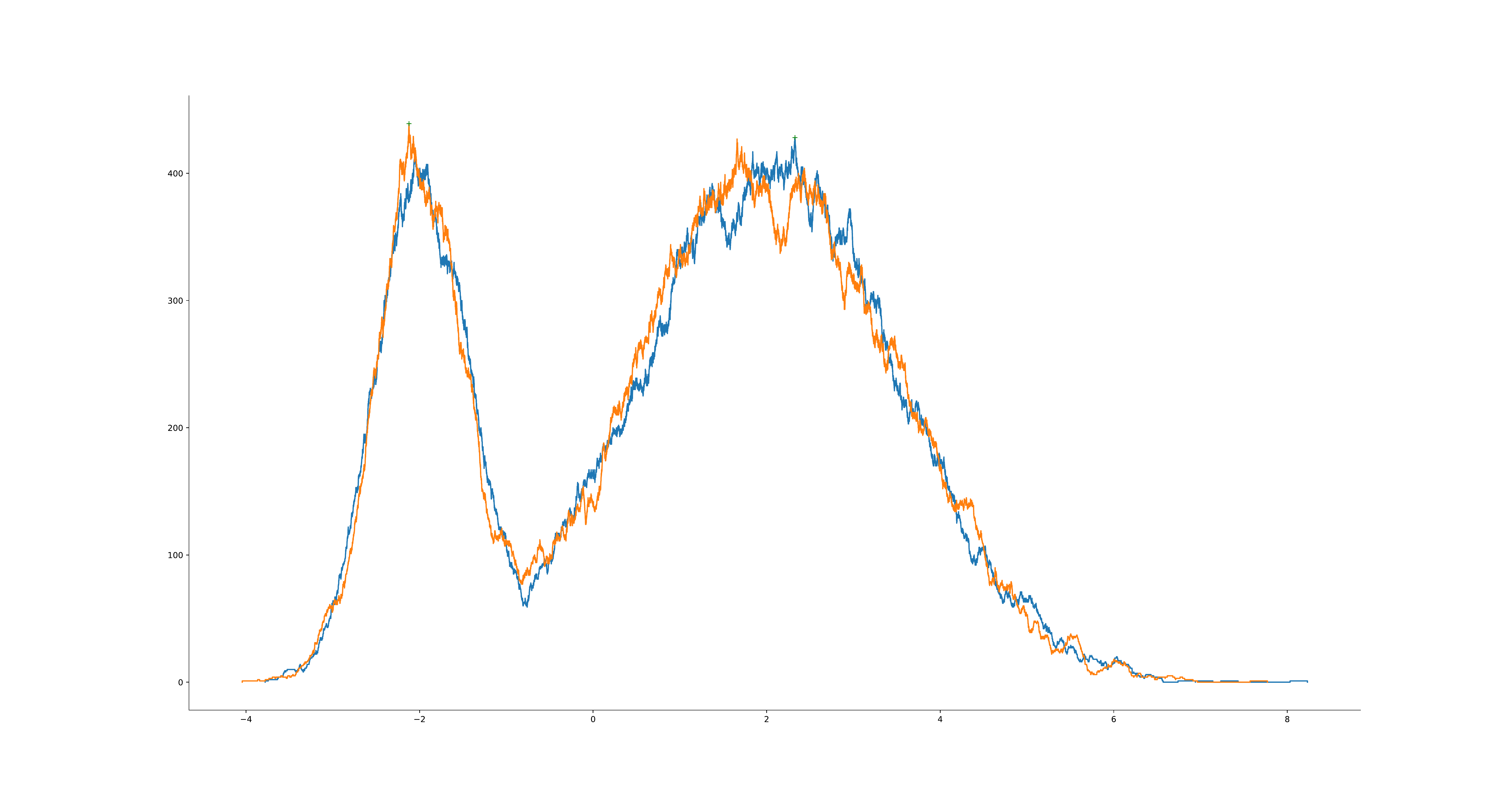}\vspace*{-20pt}}
  \caption{%
    The number $c(x)$ of the $n=10,000$ data points within $\varepsilon=0.1$ of a given point $x$, plotted here for two typical samples drawn independently from density $\heinrichp$.
    The empirical loss $\hat L(r) = \tfrac 1 n \sum_{i=1}^n L_\varepsilon(r,y_i)$ can be written $\hat L(r) = 1 - \tfrac 1 n c(r)$, and thus the maximum of each $c(\cdot)$, marked `+', corresponds to the estimated modal midpoint.
    In orange, the estimated mode and modal interval are both reasonably accurate, and in blue, neither are accurate.}
  \label{fig:modal-interval}
\end{figure}

Yet we observe that, as the midpoint $x_\varepsilon=\Gamma_\varepsilon(\heinrichp)$ of the true modal interval is very close to $m_1$ for sufficiently small $\varepsilon$, this simulation study also shows that the sample modal midpoint $\hat x_\varepsilon$ is an ineffective estimate of the true modal midpoint $x_\varepsilon$.
When $m_1$ is replaced by $x_\varepsilon$ in the preceding paragraph, we obtain qualitatively similar results: the majority of the time, $\hat x_\varepsilon$ is closer to $m_0$ than the true modal midpoint $x_\varepsilon$, and the situation worsens for $\varepsilon < 0.1$.
(See Figure~\ref{fig:modal-interval} and Appendix~\ref{app:experiments} for details.)
In summary, not only does the modal midpoint fail to estimate the mode, it fails to estimate the modal midpoint.

These empirical findings suggest the difficulty of eliciting modal midpoints in practice, despite the fact that they are elicitable.
This sentiment is confirmed by the following Corollary, which extends our argument on the elicitation complexity of the mode to modal midpoints.
The result essentially follows from the following observation.
For a distribution consisting of disjoint bump functions in $\P_{\psi, \varepsilon}$, as defined after eq.~\eqref{eq:bump} and used in the argument of Theorem~\ref{thm:bounded-density}, the mode and modal midpoint $\Gamma_\varepsilon$ coincide.
While this equivalence does not hold anymore for mixtures of Gaussians, we remark that the proof of Theorem~\ref{thm:gaussians} could be directly modified, by enlarging the width of the balls to $B_{2\sigma}(x_i)$, and choosing sufficiently small $\gamma$ and large $C$, so that the same logic would hold for the modal interval when $\varepsilon$ is sufficiently small. 

\begin{corollary}\label{cor:modal} For any $\varepsilon>0$, the modal midpoint, $\Gamma_\varepsilon:\P_\varepsilon\to \mathcal{R}$, has infinite identifiable elicitation complexity with respect to $\P_\varepsilon$, the class of probability measures defined on the real line which admit a smooth and bounded density, and have a unique mode and $\varepsilon$-modal midpoint\@.
\end{corollary}

\begin{proof}
Let $\varepsilon>0$ be given\@.
Observe that for any $p\in\subP$, the disjoint bump functions comprising $p$ are spaced far enough apart so that an interval of width $2\varepsilon$ can only intersect the support of at most one bump function\@.
Moreover, if the interval intersects the $i$th such function, it maximizes the contained mass by centering the interval at exactly $4\varepsilon i$\@.
The global maximum mass is therefore achieved by centering the interval to capture the mass of the bump function with the largest weight, whose midpoint coincides with the mode\@.
From these observations, we conclude $\mode(p)=\Gamma_\varepsilon(p)$ for all $p\in\subP$\@.
In other words, $\Gamma_\varepsilon$ and $\mode$ are the same functional with respect to $\subP\subseteq \P_\varepsilon$\@.
Hence, the identification complexity and identifiable elicitation complexity of the modal midpoint $\Gamma_\varepsilon$ with respect to $\P_\varepsilon$ are at least that of the mode\@.
\end{proof}

The fact that modal midpoints are elicitable yet have infinite identifiable elicitation complexity illustrates the subtlety of our definitions.
This subtlety is important;
as pointed out by~\citet{frongillo_kash}, one can construct pathological yet elicitable properties, such as a bijective $\Gamma:\Delta(\Y)\to[0,1]$ for finite~$\Y$ via any strictly proper scoring rule~\citep{gneiting2007strictly}.
Hence, the restriction to identifiable intermediate properties, or some other class of properties ruling out such pathologies (see Section~\ref{dis}), is necessary for practical estimation schemes.
In this light, our results are in line with the observation that both the mode and modal midpoint fail to be continuous in even weak senses: for certain distributions $p_1,p_2$, $\mode(\lambda p_1 + (1-\lambda) p_2)$ is not continuous in $\lambda$, and the same is true of $\Gamma_\varepsilon$.

To close, it is interesting to contrast the above demonstration and negative result with the existing positive results in the literature on the estimation of the mode and modal midpoints.
Some positive results, showing favorable error bounds, assume that the true density is not only unimodal but has a unique local maximum, i.e., the density increases before the mode and decreases afterwards; see for example~\citet[Section 2]{robertson1974iterative} and \citet[Assumption 2]{lee1989mode}.
Moreover, many proposed estimators are expressed as sequences of estimators which depend on the sample size~\citep{parzen1962estimation,chernoff1964estimation,grenander1965some,venter1967estimation}; we may roughly view these estimators as intermediate properties of countably infinite dimension, consistent with our results.

\section{Discussion}\label{dis}

Several interesting open questions remain\@.
One could further ask for the identifiable elicitation complexity of the mode with respect to other classes of probability distributions.
One interesting class would be distributions with densities having a unique local maximum, though note that the elicitability of the mode is still open in this case\@.
The method of perturbing the heights of (in this case, heavily overlapping) bumps as in Lemma~\ref{lem:not-iden} and Theorems~\ref{thm:bounded-density} and~\ref{thm:gaussians} does not seem sufficient for this class\@.

Another set of questions arises when stepping away from the class of identifiable properties and considering other classes, such as  weakly identifiable properties; negative results with respect to this class would show infinite complexity with respect to continuous, non-locally-constant, component-wise elicitable properties~\citep{lambert18,steinwart}\@.
Another interesting class of properties in this context would be those elicited by convex loss functions, as these properties are of practical interest yet need not be identifiable~\citep{frongillo_kash}\@.
Finally, we suspect that our techniques could be applied to other properties whose elicitation complexity is not known, such as the width of the smallest confidence interval\@.
\appendix
\section{Omitted Proofs}\label{app:proofs}

\begin{proof}[of Theorem~\ref{thm:bounded-density}]
Let $\varepsilon>0$ be given\@.
Since the identification complexity lower bounds the identifiable elicitation complexity of the mode it suffices to show that the mode is not $k$-identifiable for arbitrary $k\in \N$\@.
Suppose, by way of contradiction, that the mode is $k$-identifiable\@.
Hence, there exists a property $\hat{\Gamma}:\P\to \hat{\mathcal{R}}\subseteq \R^k$ identified by $V:\hat{\mathcal{R}}\times \R\to \R^k$ and function $f:\hat{\mathcal{R}}\to\mathcal{R}$ such that $\mode=f\circ \hat{\Gamma}$\@.
Our goal will be to specify two densities $p,p'\in\subP\subseteq \P$ with $\hat{\Gamma}(p)=\hat{\Gamma}(p')$ and $\mode(p)\neq \mode(p')$, contradicting the existence of $f$\@.

Let $t>k$ and consider the following density $p=\sum_{i=0}^t h_i\psi_{4i\varepsilon,\varepsilon}$ in $\subP$ with strictly decreasing heights $h_0>h_1>\dots h_t>h_0/2>0$ and $\sum_{i=0}^t h_i=1$\@.
Observe that $\mode(p)=0$ and denote $\hat{\Gamma}(p)=r$\@.
Consider the $k\times t$ matrix
\begin{equation}
    M= \begin{bmatrix} E_{\psi_{4\varepsilon,\varepsilon}}(V(r, Y)), \dots, E_{\psi_{4t\varepsilon,\varepsilon}}(V(r, Y))\end{bmatrix}.
\end{equation}
Let ${h}'=(h_1',\dots, h_t')$ denote a nontrivial vector in the kernel of $M$\@.
To complete the proof, we will demonstrate that for any ${h}'$ there exist real numbers $\alpha,\beta\in\R$ so that $p'=\beta\left(p+\alpha\left(\sum_{i=1}^t h_i' \psi_{4i\varepsilon,\varepsilon}\right)\right)$ is a density satisfying $\hat{\Gamma}(p')=r$ and $\Gamma(p')\neq 0$\@.
We proceed by considering all cases of ${h}'$ and showing the existence of $\alpha$ in each case\@.

First, considering $h_1',\dots ,h_t'\geq0$, let $h_{i(\text{max})}'$ denote the entry of ${h}'$ with greatest magnitude (if not unique, choose the entry associated with the maximal initial height $h_{i(\text{max})}$), and take $\alpha>(h_0-h_{i(\text{max})})/h_{i(\text{max})}'$\@.
Second, if $h_1', \dots, h_t'\leq 0$, then take $-{h}'$ and treat as above\@.
In the final case, at least one pair of entries of ${h}'$ have opposite sign\@.
Let $h_{i(\text{max})}'$ denote an entry of ${h}'$ with the greatest magnitude (if not unique, choose the entry associated with the maximal initial height $h_{i(\text{max})}$) and assume $h_{i(\text{max})}'>0$; otherwise take $-{h}'$\@.
Choose $\alpha$ such that $ (h_0-h_{i(\text{max})})/h_{i(\text{max})}'<\alpha \leq \min_{\{i: h_i'<0\}} h_i/|h_i'|$ satisfying $\alpha\neq |h_{i(\text{max})}-h_i|/(h_{i(\text{max})}'-h_i')$ for any $i$ with $h_{i(\text{max})}'>h_i'>0$\@.
Note this interval is nonempty because $h_0-h_{i(\text{max})}<\frac{h_0}{2}<h_i$ and $|h_i'|<h_{i(\text{max})}'$ for all $i$ such that $h_i'<0$\@.
In each of the above cases, there are finitely many $\alpha$ which do not yield a unimodal $p'$\@.
If the $\alpha$ chosen yields a $p'$ which is not unimodal, then discard this particular $\alpha$ from the interval and choose again\@.

With the appropriate normalization constant $\beta$, we now have a density given by $p'=\beta\left(p+\alpha\left(\sum_{i=1}^t h_i' \psi_{4i\varepsilon,\varepsilon}\right)\right)$\@.
As $h'$ is contained in the kernel of $M$, linearity of expectation and the definition of $V$ now guarantee that $\hat{\Gamma}(p')=\hat{\Gamma}(p)=r$, and the method with which we showed $\alpha$ exists ensures that $p'$ is unimodal with $\mode(p')\neq \mode(p)=0$\@.
These two statements together contradict the existence of the link function $f$ satisfying $\mode=f\circ \hat{\Gamma}$\@.
\end{proof}

\begin{proof}[of Theorem~\ref{thm:gaussians}]
As in the proof of Theorem~\ref{thm:bounded-density}, we assume the mode is $k$-identifiable and arrive at a contradiction\@.
Hence, we assume there exists a property $\hat{\Gamma}:\QQ\to \hat{\mathcal{R}}\subseteq \R^k$ identified by $V:\hat{\mathcal{R}}\times \R\to \R^k$ and function $f:\hat{\mathcal{R}}\to\mathcal{R}$ such that $\mode=f\circ \hat{\Gamma}$\@.
We will again specify two densities from $\QQ$ in the same level set of $\hat{\Gamma}$, but different modes which contradicts the existence of $f$\@.

Let $t>k$, and let $q_0,q_1,\dots, q_t$ be Gaussian densities with unit height ($\sigma^2=\frac{1}{2\pi}$) centered at $x_i=Ci$ for some $C$ to be determined.
For any mixture parameters $h=(h_0,h_1,\dots, h_t)\in\reals^{t+1}_+$, we will denote the Gaussian mixture density as follows,
\begin{equation*}
q[h](x)=\sum_{i=0}^t h_i q_i(x)\in\QQ'~,
\end{equation*}
where we define $\QQ'$ to be all positive scalings of densities in $\QQ$.
As we are interested in the mode, we can always renormalize to obtain a distribution in $\QQ$ with the same mode.
In the following, we extend $\mode(p)$ for unnormalized densities in the natural way.

Observe that for any mixture $h$, we have $\mode(q[h]) \in \cup_{i=0}^t B_{\sigma}(x_i)$, for any $C > 0$.
This follows from second-order optimality conditions: as the inflection point of a Gaussian density $N(\mu,\sigma)$ is at $\mu\pm\sigma$, we have $\frac {d^2}{dx^2} q_i(x) < 0 \iff |x-x_i| < \sigma$, and thus $\frac {d^2}{dx^2} q[h](x) < 0 \implies |x-x_i| < \sigma$ for some $i$.
Let $\gamma := q_1(\sigma) = e^{-\pi(\sigma-C)^2}$.
We will want $\gamma<\frac{1}{4(t+1)}$, and thus we choose any $C>\sigma+\sqrt{\frac{\log(4(t+1)}{\pi}}$.

We will additionally use the following claims in our proof.

\begin{myclaim}
  For all $h,i$, $h_i \leq q[h](x_i) \leq \max_{x\in B_\sigma(x_i)} q[h](x) \leq h_i + \gamma\sum_{j\neq i} h_j$.
\end{myclaim}
\begin{myclaim}
  If $h_i > \max_{j\neq i} h_j + \gamma \sum_{k} h_k$, then $\mode(q[h])\in B_\sigma(x_i)$.
\end{myclaim}
\begin{myclaim}
  If $h_i < h_j - \gamma \sum_{k\neq i} h_k$, then $\mode(q[h])\notin B_\sigma(x_i)$.
\end{myclaim}
In Claim 1, the first two inequalities are trivial, and the third follows from the observation that the contribution of $q_j$ to $q[h](x)$ is upper bounded by $h_j\gamma$ for all $x\in B_\sigma(x_i)$.
Claim 2 then follows from Claim 1: for all $j$ we have $q[h](x_i) \geq h_i > h_j + \gamma\sum_k h_k \geq h_j + \gamma\sum_{k\neq j} h_k \geq \max_{x\in B_\sigma(x_j)} q[h](x)$.
Similarly, for Claim 3,
$\max_{x\in B_\sigma(x_i)} q[h](x) \leq h_i + \gamma\sum_{k\neq i} h_k < h_j \leq q[h](x_j)$.

Finally, we construct our initial mixture $h$ so that $\sum_ih_i = 1$ and the following condition holds,
\begin{equation}
h_0-\gamma>h_1>h_2>\cdots>h_t>\frac{3}{4}h_0~.\label{eq:h-condition}
\end{equation}
% \end{enumerate}
By Claim 2, we therefore would have $\mode(q[h]) \in B_\sigma(x_0)$.
Condition~\eqref{eq:h-condition} can be satisfied for $t>5$ (and smaller if $C$ is larger); we give one explicit construction here.
Letting $c=1/(t+1)$ for ease of notation, we may take $h_0 = (5/4)c$ and $h_1=c$.
Enforcing $\sum_i h_i = 1$, the average of the remaining elements is then $c -(1/4)c/(t-1) = (1-1/4(t-1))c$ which is strictly less than $h_1$ but strictly greater than $c(1-1/16) = (3/4)h_0$, as desired.
We may therefore choose the remaining elements to be any decreasing sequence in the interval $(3h_0/4,h_1)$ whose average is $c(1-1/4(t-1)) \in (3h_0/4,h_1)$.

Now let $\hat{\Gamma}(q[h])=r$\@.
Consider the $k\times t$ matrix
\begin{equation}
  M= \begin{bmatrix} E_{q_1}(V(r, Y)) , \dots, E_{q_t}(V(r, Y))\end{bmatrix}.
\end{equation}
Let ${h}'=(h_1',\dots, h_t')$ denote a nontrivial vector in the kernel of $M$\@.
To complete the proof, we will demonstrate that for any such $h'$ there exists a real number $\alpha\in \R$ so that $q[h+\alpha h']=q[h]+\alpha \sum_{i=1}^t h_i'q_i$ (after normalization to obtain the corresponding element in $\QQ$) is the desired density\@.
We proceed by cases on the entries of $h'$\@.

First, if $h_1',\dots ,h_t'\geq0$, then let $h_{i(\text{max})}'$ denote the entry of ${h}'$ with greatest magnitude\@.
If $h_{i(\text{max})}'$ is not unique, then choose the entry associated with the maximal initial height $h_{i(\text{max})}$\@. 
Choose $\alpha$ such that
\begin{equation*}
\frac{h_0-h_{i(\text{max})}+\gamma}{h_{i(\text{max})}'-\gamma\left(\sum_{k\neq i(\text{max})}h_k'\right)}<\alpha.
\end{equation*}
This ensures that $h_0<(h_{i(\text{max})}+\alpha h_{i(\text{max})}')-\gamma\left(1+\alpha \sum_{k\neq i(\text{max})}h_k'\right)$ so that $\mode(q[h+\alpha h'])\not\in B_{\sigma}(x_0)$ by Claim 3.
Second, if $h_1', \dots, h_t'\leq 0$, then take $-{h}'$ and treat as above\@.

In the final case, at least one pair of entries of ${h}'$ have opposite sign\@.
Let $h_{i(\text{max})}'$ denote the entry of ${h}'$ with the greatest magnitude and assume $h_{i(\text{max})}'>0$; otherwise take $-{h}'$.
If $h_{i(\text{max})}'$ is not unique, then choose the entry associated with the maximal initial height $h_{i(\text{max})}$\@.
Choose $\alpha$ such that 
\begin{equation*}
\frac{h_0-h_{i(\text{max})}+\gamma}{h_{i(\text{max})}'-\gamma\left(\sum_{k\neq i(\text{max})}h_k'\right)}<\alpha\leq \min_{i:h_i'<0}\frac{h_i}{|h_i'|}.
\end{equation*}
Once again, the lower bound ensures that $h_0<(h_{i(\text{max})}+\alpha h_{i(\text{max})}')-\gamma\left(1+\alpha \sum_{k\neq i(\text{max})}h_k'\right)$ so that $\mode(q[h+\alpha h'])\not\in B_{\sigma}(x_0)$ by Claim 3\@. 
We bound $\alpha$ from above in this case to ensure that $q[h+\alpha h']\geq 0$, meaning we have a valid density\@.

It thus remains to verify that this interval is nonempty\@.
Take an index $i$ such that $h_i'<0$.
Note that $h_{i(\text{max})}'\geq \frac{\sum_{k\neq i(\text{max})}h_k'}{t}>\gamma\sum_{k\neq i(\text{max})}h_k'$, so that $h_{i(\text{max})}'-\gamma\left(\sum_{k\neq i(\text{max})}h_k'\right)>h_{i(\text{max})}'(1-\gamma t)>\frac{3h_{i(\text{max})}'}{4}\geq\frac{3|h_i'|}{4}$.
Also note that $\frac{h_0}{4}+\gamma<\frac{h_0}{4}+\frac{1}{4(t+1)}<\frac{h_0}{2}<\frac{3h_0}{4}\cdot \frac{3}{4}$.
Chaining these inequalities together,
\begin{align*}
\frac{h_0-h_{i(\text{max})}+\gamma}{h_{i(\text{max})}'-\gamma\left(\sum_{k\neq i(\text{max})}h_k'\right)}&< \frac{h_0-h_{i(\text{max})}+\gamma}{h_{i(\text{max})}'(1-\gamma t)}\\
&<\frac{\frac{h_0}{4}+\gamma}{h_{i(\text{max})}'(1-\gamma t)}\\
                                                                                                         &<\frac{\frac{3h_0}{4}\cdot \frac{3}{4}}{h_{i(\text{max})}'(1-\gamma t)}
  \\&
  \leq\frac{\frac{3h_0}{4}\cdot \frac{3}{4}}{\frac{3|h_i'|}{4}}%\\&
  =\frac{\frac{3h_0}{4}}{|h_i'|}%\\&
  <\frac{h_i}{|h_i'|}~.
\end{align*}
As this inequality holds for all such $i$, it holds for the minimum over $i$.

In each of the above cases, there are finitely many $\alpha$ which fail to yield a unimodal density, $q[h+\alpha h']$\@.
If the $\alpha$ chosen yields such a $q[h+\alpha h']$, discard this particular $\alpha$ and choose again\@.

Similar to the conclusion of Theorem~\ref{thm:bounded-density}, the density
$q[h+\alpha h']$ (after normalization to obtain the corresponding element in $\QQ$) gives the desired contradiction\@.
\end{proof}

\section{Experimental Details}\label{app:experiments}

So as to allow comparison with Heinrich~\citeyearpar{heinrich}, we consider a density $\heinrichp$ which is a mixture of two Gaussians; letting $p_1 = N(2,1.5)$ and $p_2 = N(-2,0.5)$, where $N(\mu,\sigma)$ denotes a Gaussian density with mean $\mu$ and standard deviation $\sigma$, we set $\heinrichp = 0.75 p_1 + 0.25 p_2$.
The true mode of $\heinrichp$ is $m_0 = \mode(\heinrichp) \approx -1.987047$,
% -1.98704667720157760579588321484
with the other local maximum occuring at $m_1 \approx 2.000000$.
The experiment performed is analogous to Heinrich~\citeyearpar{heinrich}: for each value of $\varepsilon$ as shown in Table~\ref{tab:study}, and in each of 1000 trials, we collect $n=10,000$ independent samples from $\heinrichp$, and measure the performance of the empirical modal midpiont $\hat x_\varepsilon$ relative to the true mode $m_0$ and true modal midpoint $x_\varepsilon = \Gamma_\varepsilon(\heinrichp)$.
In the case of a tie for $\hat x_\varepsilon$, we take the lowest value (which the reader will note should favor the correct value).
In sum, our results are qualitatively similar to Heinrich~\citeyearpar{heinrich}, in that the modal midpoint $\hat x_\varepsilon$ fails to estimate the mode, but we can also confirm that it fails to estimate the modal midpoint $x_\varepsilon$ as well.
Note in particular that the two ``Versus local max'' columns are identical.

\begin{table}[th]
  \centering
  \caption{The ineffectiveness of the modal midpoint as an estimate of the mode, or even of the modal midpoint itself.
    The table headings denote the following.
    $x_\varepsilon$: the true modal midpoint; MSE: mean squared error with respect to the true mode and true modal midpoint $x_\varepsilon$; Versus local max: the number of trials (out of 1000) where the estimate $\hat x_\varepsilon$ was closer to the true mode $m_0$, or true modal midpoint $x_\varepsilon$, than the other local maximum $m_1$; Minimal loss: the best empirical average loss observed in the 1000 trials.
  }
  \begin{tabular}{lS[table-auto-round,table-format=1.6]*2{S[table-format=2.2]}*2{S[table-format=3.0]}S[table-auto-round,table-format=1.3]}
    % {cccccccc} %this gave error - workaround seems to work fine    
  \toprule
    && \multicolumn{2}{c}{MSE} & \multicolumn{2}{c}{Versus local max} \\
    {$\varepsilon$} & {$x_\varepsilon$} & {Mode} & {Modal} & {Mode} & {Modal} & {Minimal loss}
    \\
    \midrule
    0.5   &-1.97669101040 & 15.88 & 15.80 &   0 &   0 & 0.7909 \\
    0.25  &-1.98499897122 & 11.06 & 11.05 & 302 & 302 & 0.8898 \\
    0.1   &-1.98673887353 &  8.75 &  8.75 & 447 & 447 & 0.9517 \\
    0.05  &-1.98697040102 &  9.00 &  9.00 & 433 & 433 & 0.9721 \\
    0.025 &-1.98702768942 &  9.12 &  9.12 & 424 & 424 & 0.9852 \\
    0.001 &-1.98704664678 &  8.84 &  8.84 & 431 & 431 & 0.9982 \\
    \bottomrule
  \end{tabular}
  \label{tab:study}
\end{table}
% c = 0.5 mse mode 15.8782487868 mse modal 15.7963231025 mode pref 0 modal pref 0 loss 0.7909
% c = 0.25 mse mode 11.0628185877 mse modal 11.0517223097 mode pref 302 modal pref 302 loss 0.8898
% c = 0.1 mse mode 8.74640222146 mse modal 8.7450849656 mode pref 447 modal pref 447 loss 0.9517
% c = 0.05 mse mode 9.00158982068 mse modal 9.00158982068 mode pref 433 modal pref 433 loss 0.9721
% c = 0.025 mse mode 9.11911627702 mse modal 9.11911627702 mode pref 424 modal pref 424 loss 0.9852
% c = 0.001 mse mode 8.84418487358 mse modal 8.84462719956 mode pref 431 modal pref 431 loss 0.9982

\begin{acknowledgements}
We thank Tobias Fissler and Jessie Finocchiaro for helpful suggestions, Jonas Brehmer for simplifying the proof of Lemma~\ref{lem:not-iden}, and Nicole Woytarowicz for her initial work on this project, including a proof of Lemma~\ref{lem:not-iden} in her B.S.\ thesis\@.
\end{acknowledgements}

\bibliographystyle{spbasic}      % basic style, author-year citations
\bibliography{ref}   % name your BibTeX data base

\begin{thebibliography}{16}
\providecommand{\natexlab}[1]{#1}
\providecommand{\url}[1]{{#1}}
\providecommand{\urlprefix}{URL }
\expandafter\ifx\csname urlstyle\endcsname\relax
  \providecommand{\doi}[1]{DOI~\discretionary{}{}{}#1}\else
  \providecommand{\doi}{DOI~\discretionary{}{}{}\begingroup
  \urlstyle{rm}\Url}\fi
\providecommand{\eprint}[2][]{\url{#2}}

\bibitem[{Chernoff(1964)}]{chernoff1964estimation}
Chernoff H (1964) Estimation of the mode. Annals of the Institute of
  Statistical Mathematics 16(1):31--41

\bibitem[{Fissler and Ziegel(2017)}]{fissler2017order}
Fissler T, Ziegel JF (2017) Order-{Sensitivity} and {Equivariance} of {Scoring}
  {Functions}. arXiv:171109628 [math, stat]
  \urlprefix\url{http://arxiv.org/abs/1711.09628}, arXiv: 1711.09628

\bibitem[{Fissler et~al(2016)Fissler, Ziegel et~al}]{fissler2016higher}
Fissler T, Ziegel JF, et~al (2016) Higher order elicitability and osband’s
  principle. The Annals of Statistics 44(4):1680--1707

\bibitem[{Frongillo and Kash(2015)}]{frongillo_kash}
Frongillo R, Kash I (2015) On elicitation complexity. In: Cortes C, Lawrence
  ND, Lee DD, Sugiyama M, Garnett R (eds) Advances in Neural Information
  Processing Systems 28, Curran Associates, Inc., pp 3258--3266

\bibitem[{Gneiting(2011)}]{gneiting}
Gneiting T (2011) Making and evaluating point forecasts. Journal of the
  American Statistical Association 106(494):746--762

\bibitem[{Gneiting and Raftery(2007)}]{gneiting2007strictly}
Gneiting T, Raftery AE (2007) Strictly proper scoring rules, prediction, and
  estimation. Journal of the American Statistical Association 102(477):359--378

\bibitem[{Grenander(1965)}]{grenander1965some}
Grenander U (1965) Some direct estimates of the mode. The Annals of
  Mathematical Statistics 36(1):131--138

\bibitem[{Heinrich(2014)}]{heinrich}
Heinrich C (2014) The mode functional is not elicitable. Biometrika
  101(1):245--251

\bibitem[{Lambert(2018)}]{lambert18}
Lambert NS (2018) Elicitation and evaluation of statistical forecasts. Preprint

\bibitem[{Lambert et~al(2008)Lambert, Pennock, and Shoham}]{lambert08}
Lambert NS, Pennock DM, Shoham Y (2008) Eliciting properties of probability
  distributions. In: Proceedings of the 9th ACM Conference on Electronic
  Commerce, ACM, pp 129--138

\bibitem[{Lee(1989)}]{lee1989mode}
Lee Mj (1989) Mode regression. Journal of Econometrics 42(3):337 -- 349

\bibitem[{Osband(1985)}]{osband}
Osband K (1985) Providing incentives for better cost forecasting. PhD thesis,
  University of California, Berkeley

\bibitem[{Parzen(1962)}]{parzen1962estimation}
Parzen E (1962) On estimation of a probability density function and mode. The
  Annals of Mathematical Statistics 33(3):1065--1076

\bibitem[{Robertson and Cryer(1974)}]{robertson1974iterative}
Robertson T, Cryer JD (1974) An iterative procedure for estimating the mode.
  Journal of the American Statistical Association 69(348):1012--1016

\bibitem[{Steinwart et~al(2014)Steinwart, Pasin, Williamson, and
  Zhang}]{steinwart}
Steinwart I, Pasin C, Williamson R, Zhang S (2014) Elicitation and
  identification of properties. In: Balcan MF, Feldman V, Szepesvári C (eds)
  Proceedings of The 27th Conference on Learning Theory, PMLR, Barcelona,
  Spain, Proceedings of Machine Learning Research, vol~35, pp 482--526

\bibitem[{Venter(1967)}]{venter1967estimation}
Venter J (1967) On estimation of the mode. The Annals of Mathematical
  Statistics 38(5):1446--1455

\end{thebibliography}

\end{document}